\documentclass[10pt,reqno]{amsart}
\usepackage{amssymb,amsmath,amscd,amsfonts,amsthm}
\usepackage[english]{babel}
\usepackage{enumerate}

 \newtheorem{thm}{Theorem}[section]
 \newtheorem{defn}[thm]{Definition}
 
 \newtheorem{prop}[thm]{Proposition} 
 \newtheorem{lem}[thm]{Lemma} 
 \newtheorem{cor}[thm]{Corollary} 

 \numberwithin{equation}{section}

\newcommand{\loc}{\mathop{\rm loc}}

\newcommand{\ds}{\displaystyle}

\begin{document}

\title[Some Calculations of~Orlicz Cohomology]
{Some Calculations of~Orlicz Cohomology \\ and Poincar\'e--Sobolev--Orlicz Inequalities}
\thanks{The second author was supported by~the~Program of~Basic Scientific Research
of~the~Siberian Branch of~the~Russian Academy of~Sciences.}

\author{Vladimir Gol$'$dshtein}
\address{Department of Mathematics,
Ben Gurion University of the Negev,
P.O.Box 653, Beer Sheva, Israel} 
\email{vladimir@bgumail.bgu.ac.il}

\author{Yaroslav Kopylov}
\address{Sobolev Institute of Mathematics, Pr.~Akad. Koptyuga 4,
630090, Novosibirsk, Russia  \\
Novosibirsk State University, ul.~Pirogova~1,
630090, Novosibirsk, Russia}
\email{yakop@math.nsc.ru}

 

\begin{abstract}
We carry out calculations of Orlicz cohomology for some basic Riemannian
manifolds (the~real line, the~hyperbolic plane, the~ball). Relationship
between Orlicz cohomology and Poincar\'e--Sobolev--Orlicz-type inequalities 
is discussed.

\vspace{2mm}

\vspace{2mm}
\end{abstract}

\maketitle







\section*{Introduction}

The~article continues the~study of~Orlicz cohomology of~Riemannian manifolds
initiated in~\cite{KopPan2015,Kop2017}. 

Orlicz cohomology is a~natural generalization of~$L_{qp}$-cohomology (for a~detailed
discussion of~$L_{qp}$-cohomology, the~reader is referred, for~example, to~\cite{GT06}).

Like Orlicz function spaces,
the Orlicz spaces~$L^\Phi$ of differential forms are a~natural nonlinear generalization
of the spaces~$L^p$. Orlicz spaces of differential forms
on domains in~$\mathbb{R}^n$ were first considered
by Iwaniec and Martin in~\cite{IM2001} and then by~Agarwal, Ding, and Nolder in~\cite{ADN2009}.
Orlicz forms on~an~arbitrary Riemannian manifold 
were apparently first examined by~Kopylov and Panenko in~\cite{KopPan2015}.

In~\cite{GT06}, Gol$'$dshtein and Troyanov demonstrated close relationship 
between $L_{qp}$-cohomology and Sobolev-type inequalities on~Riemannian manifolds
and, basing on this and some ``almost duality'' techniques, performed calculations 
of~$L_{qp}$-co\-ho\-mo\-lo\-gy for some basic manifolds. 
It turns out that, with some significant
corrections and sometimes under additional constraints on~the~$N$-functions from which
the~Orlicz cohomology is constructed, these methods prove to~be fruitful in~computing
Orlicz cohomology. 

The structure of the article is as follows: In~Section~\ref{N-func}, we recall
the main notions and necessary properties of Orlicz function spaces.
 In~Section~\ref{bancomp}, we recall some basic
information on~abstract Banach complexes. Section~\ref{orl_com} contains 
defininitions concerning Orlicz spaces of differential forms on~a~Riemannian 
manifold, Orlicz cohomology, and its interpretation in~terms 
of~Poincar\'e--Sobolev--Orlicz inequalities (Theorems~\ref{th.sobinor1a}
and~\ref{th.sobinor2a}). Then we calculate the~$L_{\Phi_1,\Phi_2}$-cohomology 
of~$\mathbb{R}$ (Section~\ref{cohom_rline}) the~hyperbolic plane 
(Section~\ref{cohom_hyp}) and the~$L_\Phi$-cohomology of~the~ball (``$L^\Phi$-Poincar\'e
inequality'', Section~\ref{sec.Bnor}).

\section{$N$-Functions and Orlicz Function Spaces}\label{N-func}

\begin{defn}
A~nonnegative function $\Phi:\mathbb{R}\to \mathbb{R}$ is called an {\it $N$-function} if

{\rm(i)} $\Phi$ is even and convex;

{\rm(ii)} $\Phi(x)=0 \Longleftrightarrow x=0$;

{\rm(iii)} $\lim\limits_{x\to 0} \frac{\Phi(x)}{x}=0;
\quad \lim\limits_{x\to \infty} \frac{\Phi(x)}{x}=\infty$.
\end{defn}

An $N$-function $\Phi$ has left and right derivatives (which can differ only on an at most
countable set, see, for instance, \cite[Theorem~1, p.~7]{RaRe91}).
The left derivative $\varphi$ of $\Phi$ is left continuous,
nondecreasing on~$(0,\infty)$, and such that $0<\varphi(t)< \infty$ for $t>0$, $\varphi(0)=0$,
$\lim\limits_{t\to\infty}\varphi(t)=\infty$. The function
$$
\psi(s) = \inf \{t>0 \,:\, \varphi(t)>s \}, \quad s>0,
$$
is called the {\it left inverse} of~$\varphi$.

The functions $\Phi, \Psi$ given by
$$
\Phi(x) = \int\limits_0^{|x|} \varphi(t) dt, \quad \Psi(x) = \int\limits_0^{|x|}
\psi(t) dt
$$
are called {\it complementary $N$-functions}.

The $N$-function $\Psi$ complementary to an $N$-function $\Phi$ can also be expressed as
$$
\Psi(y)= \sup \{ x|y| - \Phi(x) \,:\, x\ge 0\}, \quad y\in\mathbb{R}.
$$

Throughout the~article, given an~$N$-function~$\Phi:\mathbb{R}\to [0,\infty)$, we denote 
by $\Phi^{-1}$ its ``positive''  inverse $\Phi^{-1}:[0,\infty)\to [0,\infty)$.

$N$-functions are classified in accordance with their growth rates as follows:

\begin{defn}
An $N$-function $\Phi$ is said to satisfy the $\Delta_2$-condition (for all~$x$), 
which is written as $\Phi\in\Delta_2$ if there exists a~constant $K>2$ such that 
$\Phi(2x)\le K \Phi(x)$ for all~$x\ge 0$; $\Phi$ is said to~satisfy 
the~$\nabla_2$-condition (for all~$x$), which is denoted symbolically as 
$\Phi\in\nabla_2$, if there is a~constant $c>1$ such that 
$\Phi(x)\le \frac{1}{2c}\Phi(cx)$ for all~$x\ge0$. 
\end{defn}

It is not hard to~see that
an~$N$-function~$\Phi$ satisfies the the~$\nabla_2$-condition if and only if 
its dual $N$-function satisfies the~$\Delta_2$-condition.

\medskip

Henceforth, let $\Phi$ be an $N$-function and let $(\Omega,\Sigma, \mu)$ be a measure
space.

\begin{defn}
Given a~measurable function $f:\Omega \to \mathbb{R}$, we put
$$
\rho_{\Phi}(f):= \int_\Omega \Phi(f) d\mu.
$$
\end{defn}

\begin{defn}
The linear space
$$
L^\Phi= L^\Phi(\Omega) = L^\Phi(\Omega,\Sigma,\mu) = \{ f:\Omega\to\mathbb{R} ~\text{measurable}~ :
\rho_\Phi(af)<\infty ~\text{for {\it some} $a>0$}\}
$$
is called an {\it Orlicz space} on~$(\Omega,\Sigma,\mu)$.
\end{defn}

Let $\Psi$ be the complementary $N$-function to~$\Phi$.

Below we as usual identify two functions equal outside a~set of measure zero.

If $f\in L^\Phi$ then the functional $\|\cdot\|_\Phi$ (called
{\it the Orlicz norm}) defined by
$$
\|f\|_\Phi= \|f\|_{L^\Phi(\Omega)} = \sup \biggl\{ \biggl| \int_\Omega fg \, d\mu \biggr| :
\rho_\Psi(g) \le 1 \biggr\}
$$
is a seminorm. It becomes a norm if $\mu$ satisfies the {\it finite subset property}
(see~\cite[p.~59]{RaRe91}): if $A\in \Sigma$ and $\mu(A)>0$ then there exists
$B\in\Sigma$, $B\subset A$, such that $0<\mu(B)<\infty$.

The equivalent {\it gauge} (or {\it Luxemburg}) {\it norm} of a function $f\in L^\Phi$
is defined by the formula
$$
\|f\|_{(\Phi)}= \|f\|_{L^{(\Phi)}(\Omega)} =
\inf \biggl\{ K>0 \, : \, \rho_\Phi\biggl(\frac{f}{K}\biggr)\le 1 \biggr\}.
$$
This is a norm without any constraint on the measure~$\mu$ (see \cite[p.~54, Theorem~3]{RaRe91}).

\section{Banach Complexes}\label{bancomp}

Like in the~case of~$L_{q,p}$-cohomology, treated in~\cite{GT06}, we apply
some abstract facts about Banach complexes to~the~Orlicz cohomology 
of~Rimennian manifolds.
 
In~this section, we recall some definitions and assertions 
about abstract Banach complexes given in~\cite{GT06}.

\begin{defn}\label{def.ban.com}
A \emph{Banach complex} is a sequence $F^*=\{ F^k,d_k\}_{k\in\mathbb{N}}$ 
where $F^{k}$ is a Banach space and $d=d^k:F^k\rightarrow F^{k+1}$ is a~bounded 
operator with $d^{k+1}\circ d^{k}=0$.
\end{defn}

\begin{defn}
Given a Banach complex $\{F^k,d\}$, introduce the~vector spaces:
\begin{itemize}
\item $Z^k:=\ker(d:F^{k}\rightarrow F^{k+1})$ (a~closed subspace of~$F^{k}$);
\item $B^k:=$Im$(d:F^{k-1}\rightarrow F^{k})\subset Z^{k}$;
\item $H^k(F^{*}):=Z^k/B^k$ is the \emph{cohomology} of the complex
$F^{*}=\{F^k,d\}$;
\item $\overline{H}^k(F^{*}):=Z^k/\overline{B}^k$ is \emph{the reduced
cohomology} of the complex $F^{*}$;
\item $T^{k}(F^{*}):=\overline{B}^{k}/B^{k}=H^{k}/\overline{H}^{k}$ is
the \emph{torsion} of the complex $F^{*}$.
\end{itemize}
\end{defn}

As was observed in~\cite{GT06}, the~following easy assertion holds:

\begin{enumerate}[(a)]
\item $\overline{H}^{k},Z^{k}$ and $\overline{B}^{k}$ are Banach spaces;
\item The natural (quotient) topology on $T^{k}:=\overline{B}^{k}/B^{k}$
is coarse (any closed set is either empty or $T^{k}$);
\item there is a~natural exact sequence
\begin{equation*}\label{exact}
  0\rightarrow T^k\rightarrow
H^k\rightarrow\overline{H}^k\rightarrow 0.
\end{equation*}
\end{enumerate}

\begin{lem}\cite[Lemma~4.4]{GT06}\label{lem.notorsion}
For any Banach complex $\{ F^{k},d\}$, the following
are equivalent:
\begin{enumerate}[(i)]
\item $T^k=0$;
\item $\dim T_k < \infty$;
\item $H^k$ is a Banach space;
\item $B^k\subset F^k$ is closed.
\end{enumerate}
\end{lem}

\begin{lem}\cite[Proposition~4.5]{GT06}\label{abs.sob1}
The following are equivalent:
\begin{enumerate}[(i)]
  \item $H^{k}=0$;
  \item The operator $d_{k-1}:F^{k-1}/Z^{k-1}\rightarrow Z^{k}$ admits a
bounded inverse $d_{k-1}^{-1}$;
  \item There exists a constant $C_{k}$ such that for any \ $\theta\in
Z^{k}$ there is an element $\eta\in F^{k-1}$ with $d\eta=\theta$
and
$$
\|\eta\|_{F^{k-1}}\leq C_k\|\theta\|_{F^k}.
$$
\end{enumerate}
\end{lem}

\begin{lem}\cite[Propositions~4.6 and 4.7]{GT06}\label{abs.sob2} 
The following conditions~(i) and~(ii) are equivalent:
\begin{enumerate}[(i)]
  \item $T^{k}=0$.
  \item The operator $d_{k-1}:F^{k-1}/Z^{k-1}\rightarrow B^{k}$ admits a
bounded inverse $d_{k-1}^{-1}$.
\end{enumerate}
Any of these conditions implies 
\begin{enumerate}[(i)] \setcounter{enumi}{2}
  \item There exists a constant $C_{k}^{'}$ such that for any $\xi\in F^{k-1}$
 there is an element $\zeta\in Z^{k-1}$ such that
\begin{equation}
\Vert\xi-\zeta\Vert_{F^{k-1}}\leq C_{k}^{'}\Vert
d\xi\Vert_{F^{k}}.\label{eq2}
\end{equation}
\smallskip
\end{enumerate}
Moreover, if $F^{k-1}$ is a reflexive Banach space then
conditions~(i)-(iii) are equivalent.
\end{lem}

\section{Orlicz Spaces of Differential Forms and Orlicz Cohomology}\label{orl_com}

Let $X$ be a~Riemannian manifold of dimension~$n$. Given $x\in X$, denote by $(\omega(x),\theta(x))$
the scalar product of exterior $k$-forms $\omega(x)$ and $\theta(x)$ on~$T_x X$. This
gives a~function $x\mapsto (\omega(x),\theta(x))$ on~$X$.

Let $\Phi:\mathbb{R}\to \mathbb{R}$ and $\Psi:\mathbb{R}\to \mathbb{R}$
be two complementary $N$-functions.  Given a~measurable $k$-form $\omega$, we put
$$
\rho_\Phi(\omega):= \int_X \Phi(|\omega(x)|) d\mu_X.
$$
Here $d\mu_X$ stands for the volume element of the Riemannian manifold~$X$.
We will identify $k$-forms differing on a set of measure zero.

Given a (not necessarily orientable) Riemannian manifold~$X$, introduce
the space $L^\Phi(X,\Lambda^k)$ as the class of all measurable $k$-forms $\omega$
satisfying the condition
$$
\rho_\Phi(\alpha\omega)<\infty ~\text{for some}~ \alpha>0.
$$

As in the case of Orlicz function spaces, the space $L^\Phi(X,\Lambda^k)$ is endowed with
two equivalent norms: the {\it gauge norm}
$$
\|\omega\|_{(\Phi)}= 
\inf \biggl\{ K>0 \, : \, \rho_\Phi\biggl(\frac{\omega}{K}\biggr)\le 1 \biggr\}.
$$
and the {\it Orlicz norm}  ($\Psi$ is the complementary $N$-function to~$\Phi$):
$$
\|\omega\|_{\Phi} 
= \sup \biggl\{ \biggl| \int_X (\omega(x),\theta(x))\, d\mu_X \biggr|
: \rho_\Psi(\theta) \le 1 \biggr\}
$$
As in the case of function spaces, it can be proved that $L^\Phi(X,\Lambda^k)$ endowed
with one of these norms is a~Banach space.

Obviously, the gauge norm of a~$k$-form $\omega$ is nothing but the gauge norm of its modulus
function~$|\omega|$. The same holds for the~Orlicz norm (\cite[Lemma~2.1]{KopPan2015}). 

Unless otherwise specified, we endow the $L^\Phi$ spaces with the~gauge norms;
the quotient (semi)norm on each of the cohomology spaces to~be~defined below 
depends on~the~choice of~the~norms on~$L^{\Phi_I}$ and~$L^{\Phi_{II}}$ 
but the resulting topology does not.

\begin{defn}
A form $\theta\in L_{1,\loc}^{j+1}(X)$ is called the ({\it weak}) {\it differential}
$d\omega$ of $\omega\in L_{1,\loc}^j(X)$
if
$$
\int\limits_U \omega\land du
=(-1)^{j+1}\int\limits_U \theta \land u
$$
for every orientable domain $U\subset\mathop{Int} X$ and every form
$u\in D^{n-j-1}(X)$ having support in~$U$.
\end{defn}

Let $\Phi_I$ and $\Phi_{II}$ be $N$-functions. For $0\le k\le n$, put
$$
\Omega_{\Phi_I,\Phi_{II}}^k(X) =\left\{ \omega\in L^{\Phi_I}(X,\Lambda^k)
\, : \, d\omega\in L^{\Phi_{II}}(X,\Lambda^{k+1}) \right\}.
$$
This is a~Banach space with the norm
$$
\|\omega\|_{(\Phi_I),(\Phi_{II})} = \|\omega\|_{(\Phi_I)} + \|d\omega\|_{(\Phi_{II})}.
$$                                                                

Consider also the spaces
\begin{gather*}
Z_{\Phi_{II}}^k(X) = \{ \omega \in L^{\Phi_{II}}(X,\Lambda^k)
  \, : \, d\omega=0 \};
\\
B_{\Phi_I,\Phi_{II}}^k(X)
= \{ \omega \in L^{\Phi_{II}}(X,\Lambda^k) \, : \, \omega = d\beta ~\text{for some}~
\beta\in L^{\Phi_I}(X,\Lambda^{k-1}) \}.
\end{gather*}

Denote by $\overline{B}_{\Phi_I,\Phi_{II}}^k(X)$
the closure of $B_{\Phi_I,\Phi_{II}}^k(X)$
in~$L^{ \Phi_{II} }(X,\Lambda^k)$. 

\begin{defn}
The quotient spaces
$$
H_{ \Phi_I , \Phi_{II}}^k(X)
:=Z_{ \Phi_{II} }^k(X)/B_{ \Phi_I , \Phi_{II}}^k(X)
$$
and
$$
\overline{H}_{ \Phi_I , \Phi_{II}}^k(X)
:=Z_{ \Phi_{II}}^k(X)/\overline{B}_{\Phi_I,\Phi_{II}}^k(X)
$$
are called the $k$th {\it $L_{\Phi_I,\Phi_{II}}$-cohomology} and the $k$th
{\it reduced $L_{\Phi_I,\Phi_{II}}$-cohomology}
of the~Riemannian manifold~$X$, the latter cohomology
being a~Banach space.
Define the {\em $L_{\Phi_I,\Phi_{II}}$-torsion} as
$$T_{\Phi_I,\Phi_{II}}^{k}(X)
:=\overline{B}_{\Phi_I,\Phi_{II}}^{k}(X)\,/ B_{\Phi_I,\Phi_{II}}^{k}(X).
$$
\end{defn}

The torsion $T_{\Phi_I,\Phi_{II}}^{k}(X)$ can be either $\{0\}$ or
infinite-dimensional. In~fact, if $\dim T_{\Phi_I,\Phi_{II}}^{k}(X)<\infty$ then
$B_{\Phi_I,\Phi_{II}}^{k}(X)$ is closed, hence $T_{\Phi_I,\Phi_{II}}^{k}(X)=\{0\}$. In
particular, if $\dim T_{\Phi_I,\Phi_{II}}^{k}(X)\neq 0$ then 
$\dim H_{\Phi_I,\Phi_{II}}^{k}(X)=\infty$.

If $\Phi_I=\Phi_{II}=\Phi$ then we use the notations $\Omega_{\Phi}^k(X)$,
$H_{\Phi}^k(X)$, and $\overline{H}_{\Phi}^k(X)$ instead
of~$\Omega_{\Phi,\Phi}^k(X)$,
$H_{\Phi,\Phi}^k(X)$, and
$\overline{H}_{\Phi,\Phi}^k(X)$
respectively. Thus, the {\it $L_{\Phi}$-cohomology} $H_{\Phi}^k(X)$
(respectively, the {\it reduced $L_{\Phi}$-cohomology}
$\overline{H}_{\Phi}^k(X)$) is the~$k$th cohomology (respectively, the~$k$th 
reduced cohomology) of the cochain complex $\{\Omega_{\Phi}^*(X),d\}$.

In~\cite{GT06}, Gol$'$dshtein and Troyanov realized the $k$th
$L_{q,p}$-cohomology as the $k$th cohomology of some Banach complex. Here we apply this approach
to~$L_{\Phi_I ,\Phi_{II} }$-cohomology.

Fix an~$(n+1)$-tuple of $N$-functions $\mathcal{F}=\{\Phi_0,\Phi_1, \dots, \Phi_n\}$
and put
$$
\Omega_{\mathcal{F}}^k (X) = \Omega^k_{\Phi_k, \Phi_{k+1}}(X); 
$$
Since the weak exterior differential is a~bounded operator
$d:\Omega_{\mathcal{F}}^k(X)\to \Omega_{\mathcal{F}}^{k+1}(X)$,
we obtain a~Banach complex
$$
0\to \Omega_{\mathcal{F}}^0(X)
\to \Omega_{\mathcal{F}}^1(X) \
\to \dots \to \Omega_{\mathcal{F}}^k(X) \to \dots
\to \Omega_{\mathcal{F}}^n(X) \to 0.
$$
The {\it $L_{\mathcal{F}}$-cohomology} $H^k_{\mathcal{F}}(X)$ 
(respectively, the {\it reduced $L_{\mathcal{F}}$-cohomology} 
${\overline H}^k_{\mathcal{F}}(X)$) 
of~$X$ is the $k$th cohomology (respectively, the $k$th reduced cohomology) 
of the~Banach complex $(\Omega_{\mathcal{F}}^*,d)$. 

The above-defined cohomology spaces $H^k_{\mathcal{F}}(X)$ and
$\overline{H}^k_{\mathcal{F}}(X)$ in fact depend only on~$\Phi_{k-1}$ and~$\Phi_k$:
\begin{gather*}
H^k_{\mathcal{F}}(X) = H^k_{ \Phi_{k-1}, \Phi_k }(X)
= Z^k_{\Phi_k} (X) \left/ B^k_{\Phi_{k-1},\Phi_k} \right.; \\
\overline{H}^k_{\mathcal{F}}(X)
= \overline{H}^k_{ \Phi_{k-1}, \Phi_k }(X)
= Z^k_{\Phi_k} (X) \left/
\overline{B}^k_{ \Phi_{k-1}, \Phi_k} \right..
\end{gather*}

\medskip
The~results on~abstract Banach complexes by~Gol$'$dhstein and Troyanov enable us 
to~interprete Orlicz cohomology in~terms of~a~Poincar\'e--Sobolev--Orlicz type inequality for
differential forms on~a~Riemannian manifold $X$:

\begin{thm}\label{th.sobinor1a}
$H_{\Phi_I,\Phi_{II}}^{k}(X)=0$
if and only if there exists a constant $C<\infty$ such that for
any closed differential form~$\omega \in L^{\Phi_{II}}(X,\Lambda^k)$
there exists a~differential form $\theta\in L^{\Phi_I}(X,\Lambda^{k-1})$ such
that $d\theta=\omega$ and
\[
\| \theta \|_{L^{(\Phi_I)}}\le C \|\omega\|_{L^{(\Phi_{II})}}.
\]
\end{thm}

This result is an~immediate consequence of Lemma~\ref{abs.sob1}.

\begin{thm}\label{th.sobinor2a}
(A) If $T_{\Phi_{II}}^{k}(X)=0$ then there exists a constant $C'$ such that 
for any differential form $\theta \in \Omega^{k-1}_{\Phi_I,\Phi_{II}}(X)$
there exists a~closed form $\zeta\in Z_{\Phi_{I}}^{k-1}(X)$ such that
\begin{equation}\label{inconc.sobor2}
\|\theta-\zeta\|_{L^{(\Phi_I)}} \le C' \|d\theta\|_{L^{(\Phi_{II})}}.
\end{equation}

(B) Conversely, if $\Phi_{II}\in \Delta_2\cap \nabla_2$ and there exists a~constant $C'$ 
such that for any form 
$\theta \in \Omega^{k-1}_{\Phi_I,\Phi_{II}}(X)$ there exists 
$\zeta\in Z_{\Phi_I}^{k-1}(X)$ such that~(\ref{inconc.sobor2}) holds
then $T_{\Phi_I,\Phi_{II}}(X)=0$.
\end{thm}

\smallskip

\begin{proof}
Considering the~Banach complex~$\Omega^*_{\mathcal{F}}$ 
with $\mathcal{F}=\{\Phi_{II},\dots,\Phi_{II},\Phi_{I},\dots,\Phi_{I}\}$, where
$\Phi_{II}$ changes to~$\Phi_{I}$ at~the~$k$th position, we get 
$$
H^k_{\mathcal{F}}(X) = H^k_{ \Phi_{I}, \Phi_{II} }(X); \quad
\overline{H}^k_{\mathcal{F}}(X)
= \overline{H}^k_{ \Phi_{I}, \Phi_{II} }(X).
$$
Since $\Phi_I\in \Delta_2\cap \nabla_2$, the~Banach space 
$\Omega^{k-1}_{\Phi_{II},\Phi_{II}}(X)$
is reflexive. Theorem~\ref{th.sobinor2a} now stems from Lemma~\ref{abs.sob2}.
\end{proof}

\section{The $L_{\Phi_1,\Phi_2}$-Cohomology of~$\mathbb{R}$}\label{cohom_rline}

Let $\Phi_1$ and $\Phi_2$ be $N$-functions.

\begin{prop} \label{prop.torsion_R}
 $T_{\Phi_1,\Phi_2}^{1}(\mathbb{R})\neq 0$.
\end{prop}

\begin{proof} 
Suppose on~the~contrary that $T_{\Phi_1,\Phi_2}^{1}(\mathbb{R})= 0$. In~accordance
with Theorem~\ref{th.sobinor2a}, then there is a~Sobolev
inequality for~functions on~$\mathbb{R}$
\begin{equation}\label{sobR_orl}
\inf_{z \in \mathbb{R}} \|f-z\|_{(\Phi_1)} \leq C  \|f'\|_{(\Phi_2)}
\end{equation}
for some real positive constant $C$.

Consider the~function
$$
\theta(x)= \omega_{1/2}(x) 
= \begin{cases} 
C e^{-{\frac{1}{1-4x^2}}} & \mbox{if} \quad |x|\le 1/2, \\
0 & \mbox{if} \quad |x|> 1/2.
\end{cases}
$$
Here the~constant~$C$ is chosen so that 
$$
\int_{-\infty}^\infty \theta(x)\, dx 
= \frac{c}{2} \int_{-1}^1 e^{-{\tfrac{1}{1-t^2}}} \, dt = 1.
$$

Now, consider the~family of~smooth functions with~compact support 
$\{f_a:\mathbb{R} \to \mathbb{R} : a>0\}$, where 
$$
f_a(x) = \int_{-\infty}^x \left( \theta\left(x+\frac{3}{2}\right) 
+ \theta\left(-x+a+\frac{1}{2}\right) \right) \, dx
$$
(we owe this construction to~\cite[pp.~8--9]{Dol}).
Then $f_a(x)=1$ if $x\in [1,a]$,
$f_a(x)=0$ if $x\not\in [0,a+1]$, and $\|f_a'\|_{L^{\infty}}=: L<\infty$. Clearly,
$\|f_a-z\|_{(\Phi_1)}$ is finite only for~$z=0$. Estimate the~Orlicz norms involved 
in~(\ref{sobR_orl}). We have
$$
\rho_{\Phi_1} \left( \frac{f_a}{K} \right) 
= \int_{-\infty}^{\infty}  \Phi_1 \left( \frac{f_a(x)}{K} \right) dx
\ge \int_1^a \Phi_1 \left( \frac{1}{K} \right) dx = (a-1) \Phi_1\left(\frac{1}{K}\right). 
$$
If $\rho_{\Phi_1} \left( \frac{\Phi_1}{K} \right) \le 1$ then
$(a-1)\Phi_1\left(\frac{1}{K}\right)\le 1$, which is equivalent to 
$$
K\ge \frac{1}{\Phi_1^{-1} \left( \frac{1}{a-1} \right)}.
$$
Hence, 
$$
\|f_a\|_{(\Phi)} 
= \inf \left\{ K : \rho_{\Phi_1} \left( \frac{f_a}{K} \right) \le 1 \right\}
\ge \frac{1}{\Phi_1^{-1} \left( \frac{1}{a-1} \right)}.
$$
On~the~other hand,
\begin{multline*}
\rho_{\Phi_2}\left( \frac{f'_a}{K} \right) 
= \int_{-\infty}^{\infty}  \Phi_2 \left( \frac{f'_a(x)}{K} \right) dx
\\
= \int_0^1 \Phi_2 \left( \frac{f'_a(x)}{K} \right) dx
+ \int_a^{a+1} \Phi_2 \left( \frac{f'_a(x)}{K} \right) dx  
\le 2 \Phi_2 \left( \frac{L}{K} \right).
\end{multline*}
We have
$$
2 \Phi_2 \left( \frac{L}{K} \right) \le 1 \Longleftrightarrow 
K\ge \frac{L}{\Phi_2^{-1} \left( \frac{1}{2} \right)}.
$$
Put 
$\mathcal{M}_{f',a} = \left\{ K : \rho_{\Phi_2}\left( \frac{f'_a}{K} \right) 
\le 1\right\} $.
We have shown that if $K\ge \frac{L}{\Phi_2^{-1} \left( \frac{1}{2} \right)}$ then
$K\in \mathcal{M}_{f',a}$. Therefore,
$$
\|f'_a\|_{(\Phi_2)} 
= \inf \mathcal{M}_{f',a} \le \frac{L}{\Phi_2^{-1} \left( \frac{1}{2} \right)}.
$$
Thus, 
$$
C \ge \frac{\Phi_2^{-1}  \left( \frac{1}{2} \right)}
{L \Phi_1^{-1} \left( \frac{1}{a-1} \right)} \to \infty \quad \text{as $a\to\infty$}. 
$$
The obtained contradiction proves the~proposition.
\end{proof}

\begin{cor}
If $\Phi_1$ and $\Phi_2$ are $N$-functions then 
the~space $H^1_{\Phi_1,\Phi_2}(\mathbb{R})$ is not separated; in~particular,
$H^1_{\Phi_1,\Phi_2}(\mathbb{R}) \ne 0$.
\end{cor}

\begin{prop} 
If $\Phi_1$ and $\Phi_2$ are $N$-functions and $\Phi_2\in\Delta_2$ then 
$\overline{H}^1_{\Phi_1,\Phi_2}(\mathbb{R}) = 0$.
\end{prop}

\smallskip

\begin{proof}
Let $\omega=a(x)dx\in L^{\Phi_2}(\mathbb{R})$. For~each~$n$, put
$$
C_m= \int_{-m}^m a(x) \, dx.
$$
If $C_m=0$ then put $\lambda_m(x)\equiv 0$ for all $x\in\mathbb{R}$. 
If $C_m\ne 0$ then put
$$
\lambda_m(x)
= \mathop{\rm sign} C_m \varepsilon_m \, 
\chi \left[ -\frac{|C_m|}{2\varepsilon_m}, \frac{|C_m|}{2\varepsilon_m} \right], 
$$
where $\varepsilon_m= t_m/m$ and $t_m$ is the~only root of~the~equation 
$$
\frac{\Phi_2(t_m)}{t_m} = \frac{1}{m|C_m|}.
$$
(The function $t\mapsto \Phi_2(t)/t$ is strictly increasing; see, for~example,
\cite{KraRu}). We obviously have
$$
\int_{\mathbb{R}} \lambda_m(x)\,dx = C_m = \int_{-m}^m a(x) \, dx.
$$ 
Compute the~norm $\|\lambda_m\|_{(\Phi_2)}$. We have
$$
\rho_{\Phi_2}\left(\frac{\lambda_m}{K}\right) 
= \int_{-|C_m|/2\varepsilon_m}^{|C_m|/2\varepsilon_m} 
\Phi_2\left( \frac{\varepsilon_m}{K} \right) \, dx
= \frac{|C_m|}{\varepsilon_m} \Phi_2\left( \frac{\varepsilon_m}{K} \right).
$$
Thus, 
\begin{multline*}
\rho_{\Phi_2}\left(\frac{\lambda_m}{K}\right) \le 1 \Longleftrightarrow
\frac{|C_m|}{\varepsilon_m} \Phi_2\left( \frac{\varepsilon_m}{K} \right) \le 1
\Longleftrightarrow
\Phi_2\left( \frac{\varepsilon_m}{K} \right) \le \frac{\varepsilon_m}{|C_m|}  \\
\Longleftrightarrow
\frac{\varepsilon_m}{K} \le \Phi_2^{-1}\left( \frac{\varepsilon_m}{|C_m|} \right) 
\Longleftrightarrow K\ge \frac{\varepsilon_m}{\Phi_2^{-1} 
\left( \frac{\varepsilon_m}{|C_m|} \right) }.
\end{multline*}
Here $\Phi_2^{-1}$ stands for~the~inverse function to~$\Phi_2:[0,\infty)\to [0,\infty)$.
Hence, $\|\lambda_m\|_{(\Phi_2)}=\frac{\varepsilon_m}{\Phi_2^{-1} (\varepsilon_m/|C_m|)}$.
By~the~choice of~$\varepsilon_m$, 
$$
\frac{\Phi_2(m\varepsilon_m)}{m\varepsilon_m} = \frac{1}{m|C_m|},
$$
and so $\|\lambda_m\|_{(\Phi_2)}=\frac{1}{m}$.

Let $b_{m}(x):= 
\int_{-\infty}^{x}\left(\chi_{[-m,m]}(t)a(t)-\lambda_{m}(t)\right)dt$. Since $b_m$ has
compact support, $b_{m}\in L^{\Phi_1}(\mathbb{R})$ for~each~$m$. Furthermore,
$\| db_{m}-\omega \|_{(\Phi_2)}\le \|a\|_{L^{(\Phi_2)}(\mathbb{R}\setminus[-m,m])}+
\|\lambda_m\|_{L^{(\Phi_2)}(\mathbb{R})}\to 0$ as $m\to\infty$ since for $\Phi_2\in\Delta_2$
all functions in~$L^{\Phi_2}$ have absolutely continuous norm (\cite[Theorem~10.3]{KraRu}). 
Thus, $\overline{H}^1_{\Phi_1,\Phi_2}(\mathbb{R}) = 0$.
\end{proof}

\medskip

All the results of this section are also valid for the half-line $\mathbb{R}_{+}$ 
(with similar proofs).

\section{The $L_{\Phi_1,\Phi_2}$-Cohomology of~the~Hyperbolic Plane}\label{cohom_hyp}

We will need the~following Orlicz versions of~Propositions~8.3 and 8.4 in~\cite{GT06},
which are proved in~absolutely the~same manner:
\smallskip

\begin{prop}\label{wedge.compl1} 
Let $M$ be a~complete manifold of~dimension~$n$ and let $(\Phi_1,\Psi_1)$ 
and $(\Phi_2,\Psi_2)$ be two pairs of~complementary Orlicz functions. 
Suppose that $\alpha\in Z_{\Phi_2}^{k}(X)$
and there exists a smooth closed $(n-k)$-form $\gamma$ such that 
$\gamma\in Z_{\Psi_1}^{n-k}(X)$, $\gamma\wedge\alpha\in L^1(X,\Lambda^n)$, and
$$
\int_M \gamma\wedge\alpha\neq0,
$$
then $\alpha\notin B_{\Phi_1,\Phi_2}^{k}(X)$. In particular,
$H_{\Phi_1,\Phi_2}^k(X)\ne 0$.
\end{prop}

\begin{prop}\label{wedge.compl2} 
Let $M$ be a~complete manifold of~dimension~$n$ and let $(\Phi_1,\Psi_1)$ 
and $(\Phi_2,\Psi_2)$ be two pairs of~complementary Orlicz functions. 
Suppose that $\alpha\in Z_{\Phi_2}^{k}(X)$
and there exists a smooth closed $(n-k)$-form 
$\gamma\in Z_{\Psi_1}^{n-k}(X)\cap Z_{\Psi_2}^{n-k}(X)$ such that
$$
\int_M\gamma\wedge\alpha\neq0,
$$
then $\alpha\notin\overline{B}_{\Phi_1,\Phi_2}^{k}(X)$. In particular,
$\overline{H}_{\Phi_1,\Phi_2}^k(X)\ne 0$.
\end{prop}

The~hyperbolic plane $\mathbb{H}^{2}$ is the~Riemannian manifold that can be 
modelled as the~space~$\mathbb{R}^{2}$ endowed with the~Riemannian metric 
$$
ds^2=e^{2z} dy^2 + dz^2.
$$

For an~$N$-function~$\Phi$, introduce the~condition
$$
\int_0^1 \frac{\Phi(v)}{v^2} \, dv <\infty.
\eqno{(A)}
$$
(The~upper integration limit~$1$ can be replaced by any positive number.)

\begin{thm}\label{orcoho_h2}
If $\Phi_1$ and $\Phi_2$ are $N$-functions such that their complementary
$N$-functions~$\Psi_1$ and $\Psi_2$ and the~function~$\Phi_2$ satisfy condition~$(A)$
then
$$ 
 \dim (\bar{H}_{\Phi_1,\Phi_2}^{1}(\mathbb{H}^{2}))=\infty\,.
$$
\end{thm}

We will need the~following lemma, which is in~fact Lemma~10.2 in~\cite{GT06}:

\begin{lem}\label{lem_hn} 
There exist two smooth functions $f$ and $g$ on $\mathbb{H}^{2}$ 
such that

{\rm(1)} $f$ and $g$ are nonnegative;

{\rm(2)}  $f(y,z)=g(y,z)=0$ if $z\le0$ or $|y|\ge1$;

{\rm(3)}  $df$ and $dg\in L^{r}(\mathbb{H}^{2},\Lambda^{1})$ for any $1<r\le\infty$;

{\rm(4)}  the support of $df\wedge dg$ is contained in $\{(y,z):|y|\le1\,,\,0\le z\le1\}$;

{\rm(5)}  $df\wedge dg\ge0$;

{\rm(6)}  $\int_{\mathbb{H}^{2}}df\wedge dg=1$;

{\rm(7)}  $\frac{\partial f}{\partial y}$ and
 $\frac{\partial g}{\partial y}\in L^{\infty}(\mathbb{H}^{2})$;

{\rm(8)} $\frac{\partial f}{\partial z}$ and
  $\frac{\partial g}{\partial z}$ have compact support.
\end{lem}

\medskip

We will also need the~following generalization of~item~(3) above:  

\begin{lem}\label{lem_hn_orl}
If $\Phi$ is an~$N$-function satisfying condition~$(A)$  
then $df,\,dg\in L^{\Phi}(\mathbb{H}^{2},\Lambda^{1})$.
\end{lem}

\begin{proof}
Recall the~construction of~the~functions~$f$ and~$g$ of~\cite[Lemma~10.2]{GT06}.

Choose smooth functions $h_{1}$, $h_{2}$, and $k:\mathbb{R}\rightarrow\mathbb{R}$
with the~following properties:

(1) $h_{1}$, $h_{2}$, and $k$ are nonnegative;

(2) $h_{i}(y)=0$ if $|y|\geq1$;

(3) $h'_1(y)h_{2}(y)\geq0$ and $h_{1}(y)h'_{2}(y)\leq0$ for all $y$;

(4) the support of~the~function $h'_{1}(y)h_{2}(y)-h_{1}(y)h'_2(y)$
is not empty;

(5) $k'(z)\ge 0$ for all $z$;

(6) $k(z)=1$ if $z\ge 1$ and $k(z)=0$ if $z\le 0$.

\smallskip

Then $f$ and $g$ are defined as $f(y,z):=h_{1}(y)k(z)$ and $g(y,z):=h_{2}(y)k(z)$
respectively. 

We will now prove that $df\in L^\Phi$ by modifying the~argument of~the~proof
of~\cite[Lemma~10.2]{GT06}.

Indeed,
$$
df=h_1(y)k'(z)dz+k(z)h'_1(y)dy.
$$
The first summand $h_{1}(y)k'(z)dz$ has compact support, and
the second summand $k(z)h'_1(y)dy$ is zero outside the~infinite 
rectangle $Q=\{|y|\le 1\,; z\ge 0\}$.

Choose $D<\infty$ such that $|k(z)h'_1(y)|\leq D$ on $Q$.
We have 
$$
|k(z)h'_1(y)dy|\le D\,|dy|=D\, e^{-z}.
$$
Since the~area element of~$\mathbb{H}^2$ is $dA=e^{z}dydz$, for any $a>0$ we infer
$$
\int_{\mathbb{H}^2} \Phi(a|k(z)h'_1(y)dy|) dA
\le \int_{Q} \Phi(a D e^{-z}) e^{z}\, dy\,dz 
= aD \int_0^\infty \frac{\Phi(aD e^{-z})}{aD e^{-z}} \, dz.
$$
Putting $aDe^{-z}=v$ in~the~last integral, we get
$$
aD \int_0^\infty \frac{\Phi(aD e^{-z})}{aD e^{-z}} \, dz 
= aD \int_0^1 \frac{\Phi(v)}{v^2}\, dv < \infty.
$$ 
Thus, $\rho_\Phi(a k(z) h'_1(y)\, dy)<\infty$ for any $a$.
Consequently, $k(z)h'_1 v\in L^\Phi(\mathbb{H}^2,\Lambda^1)$.
Thus, $df=h_1(y)k'(z)dz+k(z)h'_1(y)dy$ also lies 
in~$L^\Phi(\mathbb{H}^2,\Lambda^k)$. 

The~lemma is proved.
\end{proof}

\medskip

\textit{Proof of Theorem~\ref{orcoho_h2}.} 
Take the~functions~$f$ and~$g$ on~$\mathbb{H}^2$ defined in~Lemma~\ref{lem_hn} 
and consider the~$1$-forms $\alpha=df$ and $\gamma=dg$ on $\mathbb{H}^{2}$. 
Obviously, $d\alpha=d\gamma=0$. By~Lemmas~\ref{lem_hn} and~\ref{lem_hn_orl}, 
$\alpha\in L^{\Phi}$ 
for any $N$-function~$\Phi$ such that $\int_0^1 \Phi(v)/ v^2 \,dv <\infty$
and $\gamma$ is smooth and $\gamma\in L^{\Psi_1}\cap L^{\Psi_2}$ if
$\int_0^1 \Psi_1(v)/ v^2 \,dv <\infty$ and $\int_0^1 \Psi_2(v)/ v^2 \,dv <\infty$.

Since $\int_{\mathbb{H}^2} \alpha\wedge\gamma\ne 0$, 
Proposition~\ref{wedge.compl1} shows that
$\alpha\not\in\overline{B}_{\Phi_1,\Phi_2}^{1}(\mathbb{H}^{2})$.

Now, using the isometry group of $\mathbb{H}^{2}$, we obtain an~infinite 
family of~linearly independent classes in
$\overline{H}_{\Phi_1,\Phi_2}^{1}(\mathbb{H}^{2})$. \qed

\section{The $L_\Phi$-Cohomology of~the~Ball}\label{sec.Bnor}

In~this section, we prove the~``$L^\Phi$-Poincar\'e lemma'', i.e., the~vanishing
of~the~$L^{\Phi}$-cohomology of~the~unit ball $\mathbb{B}^{n}\subset\mathbb{R}^{n}$.

Since $\mathbb{B}^{n}$ has finite volume,
$H_{\Phi_1,\Phi_2}^{0}(\mathbb{B}^{n})=
\overline{H}_{\Phi_1,\Phi_2}^{0}(\mathbb{B}^{n})=\mathbb{R}$
for any $N$-functions~$\Phi_1$ and $\Phi_2$.

For the case of~$L^p$ spaces, Gol$'$dshtein, Kuz$'$minov, and Shvedov 
proved the~vanishing of~the~$L^p$-cohomology of~the~ball 
in~\cite[Lemma 3.2]{GKSh82}; for $p\ne q$, Gol$'$dshtein and Troyanov
found necessary and sufficient conditions on~$p$ and $q$ for~the~vanishing
of~the~$L^{q,p}$-cohomology of~$\mathbb{B}^n$. Their proof is based on the~following
fact, established by~Iwaniec and Lutoborski in~\cite{IL93}:

\begin{prop}\label{iwan}
For any bounded convex domain $U\subset\mathbb{R}^{n}$ and any
$k=1,2,\dots,n$, there exists an operator
$$
T=T_{U}:L_{loc}^{1}(U,\Lambda^{k})\rightarrow
L_{loc}^{1}(U,\Lambda^{k-1})
$$
with the following properties:
\begin{enumerate}
\item[(a)] $T(d\theta)+dT\theta=\theta$ (in the sense of currents);
\item[(b)] $\ds
\left|T\theta(x)\right|\leq C\int_{U}\frac{|\theta(y)|}{\;\:|y-x|^{n-1}}dy$.
\end{enumerate}
\end{prop}
\qed

We prove

\begin{cor}\label{orl_iwan} 
If $\Phi$ is an~$N$-function then the~operator~$T$ maps $L^\Phi(U,\Lambda^{k})$ 
continuously into $L^\Phi(U,\Lambda^{k-1})$.
\end{cor}

\begin{proof} 
The following Orlicz space version of~Young's inequality for~convolution holds 
(see the~proof of~Corollary~7 in~\cite[pp.~230--231]{RaRe91}\footnote{Though it 
is required in~Corollary~7 in~\cite[pp.~230--231]{RaRe91} that $\Phi\in\Delta_2$,
the~proof of~Young's inequality works for~general $N$-functions~$\Phi$.}: 
If $f\in L^\Phi$ and $g\in L^1$ then
$f*g\in L^\Phi$ and 
$$
\|f*g\|_{\Phi} \le \|f\|_\Phi \|g\|_1.
$$
Applying this inequality to $f=|\theta|$ and $g(x)=|x|^{1-n}$, we obtain the~corollary
from~Proposition~\ref{iwan}. In~the~Orlicz norms, the~norm of~the~operator~$T$ is bounded
by~$\|g\|_1$.
\end{proof}

\begin{cor}\label{cohom_orl_iwan}
The operator $T:\Omega_\Phi(U,\Lambda^{k}) \to \Omega_\Phi(U,\Lambda^{k-1})$
is bounded and $Td\omega+dT\omega=\omega$ for any $\omega\in\Omega_\Phi^{k}(U)$.
\end{cor}

Corollary~\ref{cohom_orl_iwan} gives the~following 

\begin{thm}\label{orl_poinc}
If $\Phi$ is an~$N$-function then $H_\Phi^{k}(\mathbb{B}^{n})=0$ 
for all $k=1,...,n$.
\end{thm}

\begin{proof}
Let $\omega\in Z_\Phi^{k}(\mathbb{B}^{n})$. By Corollary~\ref{cohom_orl_iwan},  
$T\omega\in L^\Phi(\mathbb{B}^{n}, \Lambda^{k+1})$. Since 
$\omega=dT\omega+Td\omega=d(T\omega)$, we conclude 
 that $[\omega]=[dT\omega]=0\in H_\Phi^k(\mathbb{B}^{n})$ 
and so $H_\Phi^k(\mathbb{B}^{n})=0$.
\end{proof}

\end{document}